\documentclass[11pt]{article}
\usepackage{graphicx}
\usepackage{amsfonts,amsmath,amsthm,amstext,amssymb,url}
\usepackage{color}
\usepackage[dvipsnames]{xcolor}
\usepackage{authblk} 
\usepackage{hyperref}
\usepackage{color}
\usepackage{tikz}
\usepackage{circuitikz}
\usepackage{pgfplots}

\def\I{\mathbb{I}}

\def\R{\mathbb{R}}
\def\C{\mathbb{C}}
\def\eqdef{\buildrel{\rm def}\over =}

\newtheorem{theorem}{\sffamily Theorem}

\newtheorem{definition}[theorem]{\sffamily Definition}
\newtheorem{remark}[theorem]{\sffamily Remark}

\newtheorem{notation}[theorem]{\sffamily Notation}

\title{Closed formulae for multiple roots of univariate polynomials through subresultants}

\begin{document}

\author[1]{Jorge Caravantes}
\author[2]{Gema M. Diaz--Toca}
\author[3]{Laureano Gonzalez--Vega}
\affil[1]{Universidad de Alcala, Madrid (Spain)}
\affil[2]{Universidad de Murcia, Murcia (Spain)}
\affil[3]{CUNEF Universidad, Madrid (Spain)}


\maketitle

\begin{abstract}
The computation of the topology of a real algebraic plane curve is greatly simplified if there are no more than one critical point in each vertical line: the general position condition. When this condition is not satisfied, then a finite number of changes of coordinates will move the initial curve to one in general position. We will show many cases where the topology of the considered curve around a critical point is very easy to compute even if the curve is not in general position. This will be achieved by introducing a new family of formulae describing, in many cases and through subresultants, the multiple roots of a univariate polynomial as rational functions of the considered polynomial involving at most one square root.

This new approach will be used to show that the topology of cubics, quartics and quintics can be computed easily even if the curve is not in general position and to characterise those higher degree curves where this approach can be used. We will apply also this technique to determine the intersection curve of two quadrics and to study how to characterise the type of the curve arising when intersecting two ellipsoids.
\end{abstract}

\section*{Introduction}
The problem of computing the topology of a real algebraic plane curve defined implicitly has received special
attention from both  Computer  Aided Geometric Design and  Symbolic Computation, independently.  For  the Computer  Aided Geometric Design community, this problem is a basic subproblem appearing often in practice when dealing with intersection problems. For the Symbolic Computation community, on the other hand, this problem  has been the motivation for many achievements  in the study of subresultants, symbolic real root counting, infinitesimal computations, etc. By a comparison between the seminal papers and  the more renewed works,  one can see how the theoretical and practical complexities of the algorithms dealing with this problem have been dramatically improved (see for example \cite{DRR2015} and \cite{KS2015}).

The computation of the topology of a real algebraic plane curve is greatly simplified if there are no more than one critical point in each vertical line: the general position condition. When this condition is not satisfied, then a finite number of changes of coordinates will move the initial curve to one in general position. We will show here many cases where the topology of the considered curve around a critical point is very easy to compute even if the curve is not in general position. This will be achieved by introducing a new family of formulae describing, in many cases and through subresultants, the multiple roots of a univariate polynomial as rational functions of the considered polynomial involving at most one square root.

This paper is divided into  five sections. The first two sections introduce the general position condition when computing the topology of a real algebraic plane curve and subresultants together with some tools to use when solving the real root counting problem. Third section introduces formulae describing the real multiple roots of an univariate polynomial in terms of their coefficients (typically rational functions involving in the worst case a square root). Forth section show how to deal with the computation of the branches around a critical point when we have an easy to deal with algebraic description: this is the alternative we propose instead of using the ``general condition" (in those cases where the strategy presented here works). Fifth section introduces an application.
\section{Computing the topology of $P(x,y)=0$: why general position?}\label{sec:topology}
The characterizacion of the topology of a curve ${\cal C}_P$ presented by the equation $P(x,y)=0$ follows a sweeping strategy usually based on the location of the critical points of $P$ with respect to $y$ (ie those singular points or points with a vertical tangent) and on the study of the half-branches of ${\cal C}_f$ around these points since, for any other point of ${\cal C}_P$, there will be only one half-branch to the left and one hal-fbranch to the right.

\begin{definition}\label{sec:points}Let $P(x,y)\in\R[x,y]$,
$${\cal C}_P=\{(\alpha,\beta)\in\R^2:P(\alpha,\beta)=0\}$$
the real algebraic plane curve defined by $P$ and $\alpha\in\R$. 
\begin{itemize}
\item A point
$(\alpha,\beta)\in\C^2$ is called a critical point of ${\cal C}_P$ if $$P(\alpha,\beta)={\frac{\partial P}{\partial y}}(\alpha,\beta)=0.$$
\item A critical point is said to be singular if
$${\frac{\partial P}{\partial x}}(\alpha,\beta)=0.$$
\item A point $(\alpha,\beta)\in\R^2$ is a regular point of ${\cal C}_P$ if $P(\alpha,\beta)=0$ and it is not a critical point.
\end{itemize}
Non singular critical points are called ramification points.
\end{definition}

The usual strategy to compute the topology of a real algebraic plane
curve $\cal C$ defined implicitly by a polynomial $P(x,y)\in
\mathbb{R}[x,y]$ proceeds in the following way:
\begin{enumerate}
\item Compute the discriminant of $P$ with respect to
$y$, $D(x)$ and its real roots, $\alpha_1<\alpha_2<\ldots <
\alpha_{r}$: the $x$--coordinates of the critical points of ${\cal C}_P$.
\item For every $\alpha_i$, compute the
real roots of $P(\alpha_i,y)$, $\beta_{i,1}<\ldots
<\beta_{i,s_i}$ and determining which $\beta_{i,j}$ are regular points and which $\beta_{i,j}$ are critical points. Each $x=\alpha_i$ is called a critical line.
\item For every $\alpha_i$ and every $\beta_{i,j}$, compute the number of half-branches to the right and
to the left of each point $(\alpha_i,\beta_{i,j})$.
\end{enumerate}

First two steps provide the vertices of a graph that will represent the topology of the
considered curve.  Figure \ref{topology} shows how all this information will allow us to determine the topology of the cutcurve once all these points have been determined (for details see \cite{BPR, LGV_IN}). This graph is very helpful when tracing the curve numerically, since we will know exactly how to proceed when coming closer to a complicated point.

\begin{figure}[hbt]
\centering
\fbox{\includegraphics[scale=0.20]{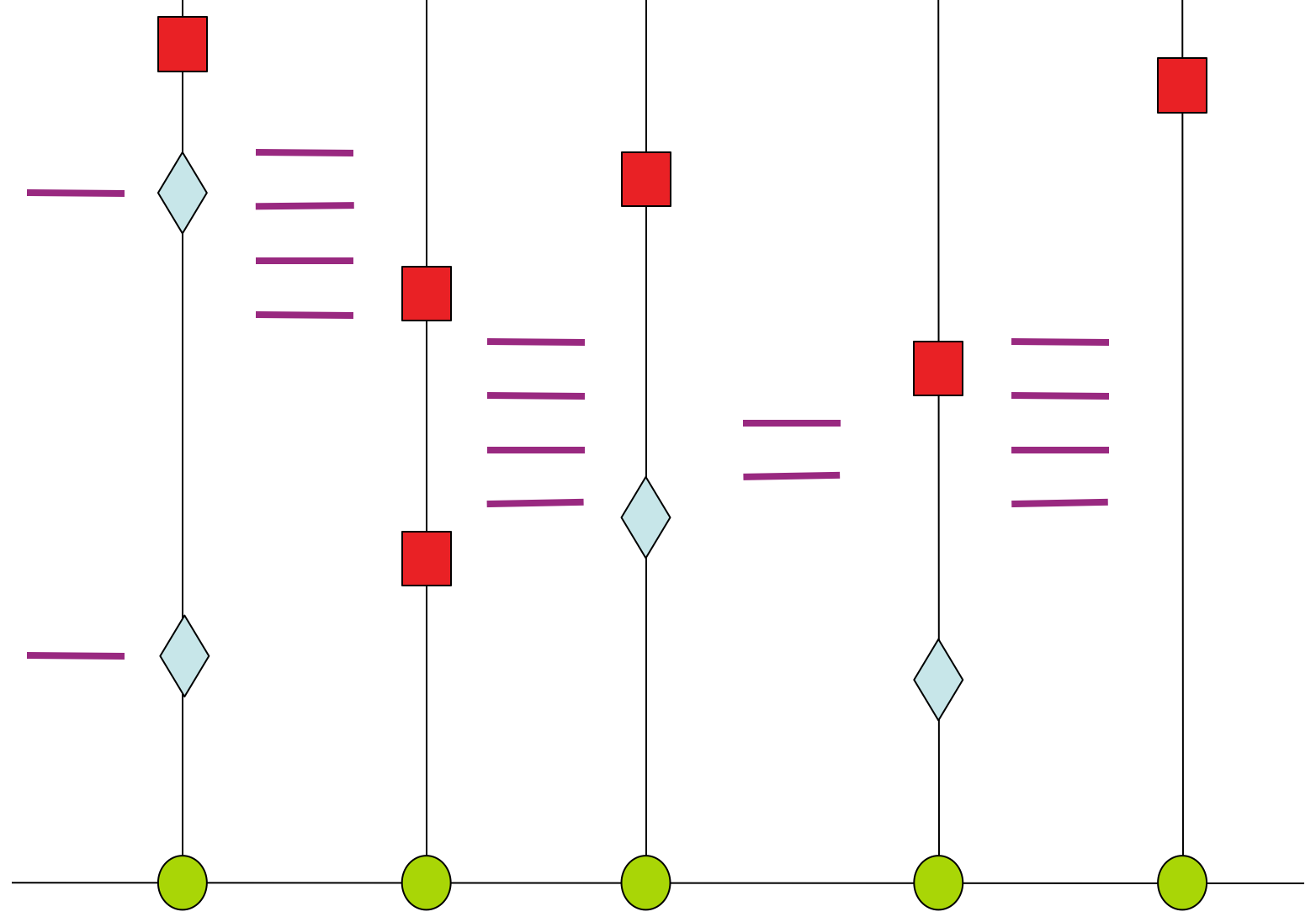}}
\caption{\label{topology} Critical points: red. Regular points in the critical lines: blue. Branches: violet. Critical points projection: green.}
\end{figure}

First step starts with the computation of the discriminant $R(x)$ of $P$ with respect to
$y$ by any available method (a determinant computation, subresultants, etc.) and typically of its squarefree part
$$D(x)=\frac{R(x)}{\gcd(R(x),R'(x))}.$$
This step ends with the computation of the real roots of $D(x)$.

In the second step, in order to avoid the numerical problems arising from the computation of the multiple roots of every $P(\alpha_i,y)$, a linear change of coordinates might simplify this and further computations. Such change of coordinates puts the curve in a desirable position which is known as the ``general position" (see \cite{LGV_MK, LGV_IN}).

\begin{definition}Let $P\in\R[x,y]$ be a squarefree polynomial. The real algebraic plane curve defined by $P$, ${\cal C}_P$, is in general position if the following  two conditions are satisfied:
\begin{enumerate}
\item The leading coefficient of $P$ with respect to $y$ (which is a polynomial in $\R[x]$) has no real roots.
\item For every $\alpha\in\R$ the number of critical points in the vertical line $x=\alpha$ is $0$ or $1$.
\end{enumerate}
\end{definition}

Two are the main advantages of having the curve in general position (see \cite{LGV_MK, LGV_IN}):
\begin{enumerate}
\item It is possible to compute rational functions $r_{\alpha_i}(x)\in\R(x)$ such that for each critical point $(\alpha_i,\beta_i)$ we have $\beta_i=r_{\alpha_i}(\alpha_i)$. This allows also to symbolically construct a squarefree polynomial $g_i(\alpha_i,y)$, from every $P(\alpha_i,y)$ (i.e., by symbolically dividing $P(\alpha_i,y)$ by a convenient power of $y-r_{\alpha_i}(\alpha)_i$) whose real roots need to be computed too.
\item The edges of the topological graph representing  ${\cal C}_P$  around each critical point $(\alpha_i,\beta_i)$ can be obtained using a straightforward combinatorial reasoning.
\end{enumerate}

If the curve is not in general position we can apply the linear change of coordinates and restart the process with the new polynomial. After a finite number of such transformations the general position of the curve is guaranteed. Notice that these changes of coordinates do not modify the topology of the curve and they can be undone at the end.

Third step is accomplished by computing the number of real roots of the squarefree
polynomials $P(\gamma_i,y)$ ($i\in\{1,2,\ldots,r+1\}$) with
$\gamma_1=-\infty$, $\gamma_{r+1}=\infty,$ and for $2\le i \le r$, $\gamma_i$ being any
real number in the open interval $(\alpha_{i-1},\alpha_{i})$. When ${\cal C}_P$ is in general position this implies that, to the right of $x=\alpha_{i-1}$, all halfbranches start from $(\alpha_i,\beta_i)$ except one halfbranch starting from each regular point of ${\cal C}_P$ in the vertical line $x=\alpha_{i-1}$ (same thing happens to the left of $x=\alpha_{i}$).

\section{GCD and Real Root counting through subresultants}\label{sec:subresultants}
Subresultants will be the algebraic tool to use to determine in a very easy and compact way the greatest common divisor of two univariate polynomials or the number of different real roots of an univariate polynomial when they involve parameters or algebraic numbers as coefficients.

\begin{definition}\label{subres}
Let  $$P(T)=\sum_{i=0}^p a_{i}T^i\qquad \hbox{and} \qquad Q(T)=\sum_{i=0}^q b_{i}T^i$$ be two polynomials with coefficients in a field with $p\geq q$ and $j\in\{0,1,\ldots,q-1\}$. Denoting 
\[\delta_k=(-1)^{\frac{k(k+1)}{2}}\]   for every integer $k$, we define the $j$--th subresultant polynomial of $P$ and $Q$ with respect to $T$ in the following way (as in \cite{Li2006a}):
\[\mbox{\bf Sres}_j(P,Q)=(-1)^j\delta_{p-j-1}
\left\vert
\begin{array}{cccccccccc}
a_p & a_{p-1}   &  a_{p-2}  &\ldots  &\ldots&a_0    &&&\\
       &\ddots&\ddots &\ddots&         &&\ddots&&\\
       &          &  a_p  & a_{p-1} &a_{p-2} &\ldots  &\ldots&a_0    &\\
 b_q& b_{q-1}   &  b_{q-2}  &\ldots  &\ldots&\ldots    &b_0&&\\
       &\ddots&\ddots &\ddots&         &&&\ddots&\\
       &          &  b_q  & b_{q-1} &b_{q-2} &\ldots  &\ldots&\ldots&b_0\\
             &          &           &       &       &    1     &-T&&\\
       &          &           &       &       &          &\ddots&\ddots&\\
       &          &           &       &       &&&    1     &-T
\end{array}
\right\vert
\!\!\!\!\!\!\!\
\begin{array}{l}
\left.\begin{array}{c}\\ \\ \\  \\  \end{array}\right\} q-j
\\
\left.\begin{array}{c}\\ \\ \\  \end{array}\right\} p-j
\\
\left.\begin{array}{c}\\ \\ \\  \\  \end{array}\right\} j
\end{array}
\]
and we define the $j$--th subresultant coefficient  of $P$ and $Q$ with respect to $T$,
$\mathrm{\bf sres}_j(P,Q)$, as the coefficient of $T^j$ in $\mathrm{\bf Sres}_j(P,Q)$.
The resultant of $P$ and $Q$ with respect to $T$ is:
$$\mathrm{\bf Resultant}(P,Q)= \mathrm{\bf Sres}_0(P,Q)=\mathrm{\bf sres}_0(P,Q)\enspace.$$ 
\end{definition}

There are many different ways of defining and computing subresultants: see \cite{BPR,Li2006a} and \cite{Gonzalez-Vega2009} for a short introduction and for a pointer to several references. The use of only one sequence of subresultants for dealing with the gcd and the real root counting problems motivates the ``unusual" introduction of the sign $(-1)^j\delta_{p-j-1}$ in the previous definition. In this way we avoid to use subresultants for gcd computations and signed subresultants or the Sturm--Habicht sequence for solving the real root counting problem.

Subresultants allow to characterize easily the degree of the greatest common divisor of two univariate polynomials whose coefficients depend on one or several parameters. Since the resultant of $P$ and $Q$ is equal to
the polynomial ${\bf sres}_{0}(P,Q)$: 
\begin{equation}\label{resultant} 
\hbox{${\bf sres}_0(P,Q)=0$ if and only if there exists $T_0$ such that $P(T_0)=0$ and  $Q(T_0)=0$.}
\end{equation}
More generally, the determinants ${\bf sres}_{j}(P,Q)$, which are the formal leading coefficients of the subresultant sequence for $P$ and $Q$, can be used to compute the greatest common divisor of $P$ and $Q$ thanks to the following equivalence:
\begin{equation}\label{gcd} {\bf Sres}_i(P,Q)=\gcd(P,Q)\quad\Longleftrightarrow\quad
\begin{cases}{\bf sres}_0(P,Q)=\ldots={\bf sres}_{i-1}(P,Q)=0& \cr
            \hfill{\bf sres}_i(P,Q)\neq 0\hfill&
\cr\end{cases}\end{equation}

\begin{equation}\label{gcd} {\bf Sres}_i(P,Q)=\gcd(P,Q)\quad\Longleftrightarrow\quad
\begin{cases}s_0(P,Q)=\ldots=s_{i-1}(P,Q)=0& \cr
            \hfill s_i(P,Q)\neq 0\hfill&
\cr\end{cases}\end{equation}

The use of subresultants for solving the real root counting problem was introduced in \cite{GLRR1} following the seminal works of W. Habicht (see \cite{Habicht}). Proofs of the results described here can be found in \cite{BPR,GLRR1,GLRR2}. Next definition introduces the subresultant sequence associated to $P$ as the subresultant sequence for $P$ and $P'$, the main tool we will use to count the number of different real roots of a univariate polynomial.

\begin{definition}\label{defSH}\hskip 0pt\\
Let $P$ be a polynomial in $\R[T]$ with $p=\deg(P)$. We define the subresultant sequence of $P$ as
${\bf Sres}_{p}(P)=P$, ${\bf Sres}_{p-1}(P)=P'$ and for every
$j\in\{0,\ldots,p-2\}$:
\[{\bf Sres}_{j}(P)={\bf Sres}_{j}(P,P').\]
For every $j$ in $\{0,\ldots,p\}$ the principal $j$--th subresultant coefficient of $P$ is defined as:
\[{\bf sres}_j(P)={\rm coef}_j({\bf Sres}_{j}(P))\]
\end{definition}

It is important to quote here that the discriminant of $P$ is equal to
the polynomial ${\bf sres}_{0}(P)$ modulo the leading coefficient, $a_p$,
of $P$: $${\bf sres}_0(P)=a_p{\rm discriminant}(P).$$ 

Sign counting on the principal subresultant coefficients
provides the number of different real roots of the
considered polynomial. Next definitions show which are the sign
counting functions to be used in the sequel (see \cite{GLRR1,GLRR2}).
\begin{definition}\hskip 0pt\\ Let $\I=\{a_0,a_1,\ldots ,a_n\}$ be a
list of non zero elements in $\R$.
\begin{itemize}
\item ${\bf V}(\I)$ is defined as the number of sign variations in
the list $\{a_0,a_1,\ldots ,a_n\}$, \smallskip \item ${\bf P}(\I)$
is defined as the number of sign permanences in the list
$\{a_0,a_1,\ldots ,a_n\}$.
\end{itemize}
\end{definition}

\begin{definition}\label{cuenta}\hskip 0pt\\ Let $a_{0},a_{1}, \ldots
, a_{n}$ be elements in $\R$ with $a_{0}\neq0$ and  with the
following distribution of zeros:
\[\I=\{a_{0},a_{1}, \ldots , a_{n}\} = \]
\[\begin{matrix}  =\{ a_{0}, \ldots , a_{i_{1}}, \overbrace{0,
\ldots , 0}^{k_{1}},  a_{i_{1}+k_{1}+1},\ldots , a_{i_{2}},
\overbrace{0, \ldots , 0}^{k_{2}},  a_{i_{2}+k_{2}+1}, ,
a_{i_{3}}, 0,.\ldots ,0, a_{i_{t-1}+k_{t-1}+1}, \ldots
, a_{i_{t}},  \overbrace{0,\ldots , 0}^{k_{t}} \} \cr\end{matrix}
\]  where all the $a_{i}$'s that have been written are not 0.
Defining  $i_{0}+k_{0}+1=0$ and:
\[{\bf C}( \I ) = \sum_{s=1}^{t} \bigl({\bf
P}(\{a_{i_{s-1}+k_{s-1}+1},\ldots ,a_{i_{s}}\})-{\bf
V}(\{a_{i_{s-1}+k_{s-1}+1},\ldots ,a_{i_{s}}\})\bigr)+
\sum_{s=1}^{t-1} \varepsilon_{i_{s}}\]  where:
\[\varepsilon_{i_{s}}=\begin{cases}
\quad\quad\quad\quad\quad 0&  \hbox{if $k_{s}$  is odd} \cr
                              \displaystyle(-1)^{\frac{k_s}{2}}{\rm
sign}\left(\frac{a_{i_s+k_s+1}}{a_{i_s}}\right)& \hbox{if   $k_{s}$  is
even}\cr\end{cases} \]
\end{definition}

Next the relation between the number of real zeros of a polynomial
$P\in\R[x]$ and the polynomials in the subresultant sequence of
$P$ is presented. Its proof can be found in \cite{GLRR1,GLRR2}.

\begin{theorem}\label{SH2}\hskip 0pt\\ If $P$ is a polynomial in $\R[T]$ with
$p=\deg(P)$ then:
\[{\bf C}(\{{\bf sres}_p(P),\ldots ,{\bf sres}_0(P)\})=\#\{\alpha\in\R\colon P(\alpha)=0\}.\]
\end{theorem}

The number of different real roots of $P$ is determined exactly by the signs of the last $p-1$ determinants ${\bf sres}_i(P)$ (being the first two ones  ${\rm lcof}(P)$ and $p\;{\rm lcof}(P)$ with ${\rm lcof}(P)$ denoting the leading coefficient of $P$). If all ${\bf sres}_j(P)$ are different from $0$ then the number of different real roots of $P$ agrees with the difference between the number of sign agreements and the number of sign changes in the list of principal subresultant coefficients of $P$. It is easy to recognise here that, in this case, this is the same than the difference between the number of sign changes in 
$\{{\bf Sres}_j(P)(-\infty)\}_{j=0,1,\ldots,p}$ and the number of sign changes in $\{{\bf Sres}_j(P)(+\infty)\}_{j=0,1,\ldots,p}$ (in the same way than when using Sturm sequences).

The definition of the polynomials in the subresultant sequence of $P$ through determinants allows to perform computations dealing with real roots in a generic way: if $P$ is a polynomial with parameters or algebraic numbers as coefficients whose degree does not change after specialisation then the subresultant sequence for $P$ can be computed without specialising the parameters and the result is always good after specialisation (modulo the condition over the degree of $P$). This is not true  when using Sturm sequences (the computation of the euclidean remainders makes to appear denominators which can vanish after specialisation) or negative polynomial remainder sequences (even fixing the degree of $P$, the sequence has not always the same number of elements:  see \cite{GLRR1,Loos} for a more detailed explanation).

\begin{notation}\label{notationsubres}\hskip 0pt\\ {\rm If $P$ is a polynomial in $\R[T]$  with $p=\deg(P)$ and $0\leq k\leq p-2$ then the coefficients of the subresultant of $P$ of index $k$ will be denoted in the following way: 
$${\bf Sres}_k(P)\eqdef s_k(P) T^k+s_{k,k-1}(P) T^{k-1}+\ldots+s_{k,1}(P) T+s_{k,0}(P)\ .$$
This definition is extended to indexes $p$ and $p-1$ by introducing:
$${\bf Sres}_p(P)\eqdef P\qquad \qquad {\bf Sres}_{p-1}(P)\eqdef P'$$
When the context shows clearly who $P$ is, we will write $s_k$ and $s_{k,j}$ instead of $s_k(P)$  and $s_{k,j}(P)$.}
\end{notation}

The results presented in this section will be mainly applied to a polynomial $P(\alpha,y)$ where $\alpha$ is a real algebraic number.

\section{Multiple real roots of univariate polynomials through subresultants}\label{sec:multipleroots}
We introduce here formulas describing the real multiple roots of an univariate polynomial in terms of their coefficients. This will be always possible for degrees 2, 3, 4 and 5 and in most cases for degrees 6 and 7. It will be also characterised when this will be possible in the general case. Recently, in \cite{Hong2020}, the multiplicity structure of an univariate polynomial has been characterised in terms of its coefficients.

The main aim of deriving these formulae is to use them to describe de $y$-coordinates of the points on a critical line of a real algebraic plane curve defined implicitely. 

Cases 2 and 3 are not considered here since they are very easy to analyse.

\subsection{$\deg(P)=4$}\label{subsec:multipleroots_4}
If $a_4\neq 0$ then the polynomial in $\R[x]$
$$P(T)=a_4T^4+a_3T^3+a_2T^2+a_1T+a_0$$
factors, when there are multiple roots, only in the following five ways:
\begin{enumerate}
\item $P(T)=a_4(T-\beta)^4$ with $\beta\in\R$.
\item $P(T)=a_4(T-\beta)^3(T-\gamma)$ with $\beta,\gamma\in\R$ and $\beta\neq\gamma$.
\item $P(T)=a_4(T-\beta)^2(T-\gamma)^2$ with $\beta,\gamma\in\R$ and $\beta\neq\gamma$.
\item $P(T)=a_4(T-\gamma)^2(T-\overline{\gamma})^2$ with $\gamma\in \C-\R$.
\item $P(T)=a_4(T-\beta)^2(T-\gamma_1)(T-\gamma_2)$ with $\gamma_1\neq \gamma_2$ (if $\gamma_1\in \C-\R$ then $\gamma_2=\overline{\gamma_1}\in \C-\R$).
\end{enumerate}
Polynomials $s_i=s_i(P)$ and $s_{i,j}=s_{i,j}(P)$ will characterise each possibility in the following way (according to (\ref{gcd})):
\begin{enumerate}
\item If $s_0=s_1=s_2=0$ then
$$\gcd\left(P,P'\right)=P'={\bf Sres}_{3}(P),\quad P(T)=a_4(y-\beta)^4\quad\hbox{and}\quad\beta=-\frac{s_{3,2}}{3s_3}=-\frac{a_3}{4a_4}.$$
\item If $s_0=s_1=0$, $s_2\neq 0$ and $s_{2,1}^2-4s_2s_{2,0}=0$ then
$$\gcd\left(P,P'\right)={\bf Sres}_{2}(P), \quad
P(T)=a_4(T-\beta)^3(T-\gamma),$$
with $\beta,\gamma\in\R$, $\beta\neq\gamma$ and
$$\beta=-\frac{s_{2,1}}{2s_2},\quad 
T-\gamma=\frac{P(T)}{a_4(T-\beta)^3}\quad\hbox{and}\quad \gamma=\frac{a_0}{a_4\beta^3}$$
when $\beta\neq 0$. If $\beta=0$ then $\gamma=a_3/a_4$ with $a_3\neq 0$.
\item If $s_0=s_1=0$, $s_2\neq 0$ and $s_{2,1}^2-4s_2s_{2,0}>0$ then
$$\gcd\left(P,P'\right)={\bf Sres}_{2}(P), \quad P(T)=a_4(T-\beta_1)^2(T-\beta_2)^2$$
with $\beta_1,\beta_2\in\R$ and
$$\beta_1,\beta_2=\frac{-s_{2,1}\pm\sqrt{s_{2,1}^2-4s_2s_{2,0}}}{2s_2}.$$
\item If $s_0=s_1=0$, $s_2\neq 0$ and $s_{2,1}^2-4s_2s_{2,0}<0$ then
$$\gcd\left(P,P'\right)={\bf Sres}_{2}(P), \quad
P(T)=a_4(T-\gamma)^2(T-\overline{\gamma})^2$$
with $\gamma\in\C-\R$.
\item If $s_0=0$ and $s_1\neq 0$ then
$$\gcd\left(P,P'\right)={\bf Sres}_{1}(P), \quad P(T)=a_4(T-\beta)^2(T-\gamma_1)(T-\gamma_2),$$
with $\gamma_1\neq \gamma_2$ (if $\gamma_1\in\C-\R$ then $\gamma_2=\overline{\gamma}\in \C-\R$) and 
$$\beta=-\frac{s_{1,0}}{s_1}\quad\hbox{and}\quad (T-\gamma_1)(T-\gamma_2)=
\frac{P(T)}{a_4(T-\beta)^2}.$$
\end{enumerate}
In all cases we have characterized the real roots, simple or multiple, of any degree $4$ polynomials with multiple roots as explicit functions of the coefficients of $P$, $a_0$, $a_1$, $a_2$, $a_3$ and $a_4$. These functions are either rational functions in the $a_i$'s or the square root of a polynomial in the $a_i$'s known to be strictly positive.

\subsection{$\deg(P)=5$}
If $a_5\neq 0$ then the polynomial in $\R[x]$
$$P(T)=a_5T^5+a_4T^4+a_3T^3+a_2T^2+a_1T+a_0$$
factors, when there are multiple roots, only in the following five ways:
\begin{enumerate}
\item $P(T)=a_5(T-\beta)^5$ with $\beta\in\R$.
\item $P(T)=a_5(T-\beta)^4(T-\gamma)$ with $\beta,\gamma\in\R$.
\item $P(T)=a_5(T-\beta)^3(T-\gamma)^2$ with $\beta,\gamma\in\R$.
\item $P(T)=a_5(T-\beta)^3(T-\gamma_1)(T-\gamma_2)$ with $\beta,\gamma_1,\gamma_2\in\R$ and $\gamma_1\neq \gamma_2$.
\item $P(T)=a_5(T-\beta)^3(T-\gamma)(T-\overline{\gamma})$ with with $\beta\in\R$ and $\gamma\in \C-\R$.
\item $P(T)=a_5(T-\beta)^2(T-\gamma_1)^2(T-\gamma_2)$ with $\gamma_1\neq \gamma_2$ and $\beta,\gamma_i\in\R$.
\item $P(T)=a_5(T-\beta)^2(T-\gamma_1)(T-\gamma_2)(T-\gamma_3)$ with $\gamma_1\neq \gamma_2\neq \gamma_3$ and $\beta,\gamma_i\in\R$.
\item $P(T)=a_5(T-\beta)^2(T-\gamma_1)(T-\gamma_2)(T-\overline{\gamma_2})$ with $\beta,\gamma_1\in\R$ and $\gamma_2\in \C-\R$.
\item $P(T)=a_5(T-\beta)(T-\gamma)^2(T-\overline{\gamma})^2$ with $\beta\in\R$ and $\gamma\in \C-\R$.
\end{enumerate}

We define $\tau_0(T)=P(T)$ and, for $k\geq 1$: 
$$\tau_k(T)=\gcd(\tau_{k-1},\tau_{k-1}').$$
Polynomials $s_i(P)$ will characterize each possibility for the greatest common divisor of $P$ and $P'$, $\tau_1(P)$, in the following way (according to (\ref{gcd})):
\begin{enumerate}
\item If $s_0(P)=s_1(P)=s_2(P)=0$ and $s_3(P)=0$ then $\tau_1(T)=P'={\bf Sres}_{4}(P)$,
$$P(T)=a_5(T-\beta)^5\quad\hbox{and}\quad\beta=-\frac{s_{4,3}(P)}{4s_4(P)}=-\frac{a_4}{5a_5}.$$
\item If $s_0(P)=s_1(P)=s_2(P)=0$ and $s_3(P)\neq 0$ then $\tau_1(T)={\bf Sres}_{3}(P)$. The only possible cases are:
\begin{enumerate}
\item $P(T)=a_5(T-\beta)^4(T-\gamma)$ with $\beta,\gamma\in\R$ and $\beta\neq\gamma$.
\item $P(T)=a_4(T-\beta)^3(T-\gamma)^2$ with $\beta,\gamma\in\R$ and $\beta\neq\gamma$.
\end{enumerate}
\item If $s_0(P)=s_1(P)=0$ and $s_2(P)\neq 0$ then $\tau_1(T)={\bf Sres}_{2}(P)$. The only possible cases are:
\begin{enumerate}
\item $P(T)=a_5(T-\beta)^3(T-\gamma_1)(T-\gamma_2)$ with $\beta,\gamma_1,\gamma_2\in\R$ and $\beta\neq\gamma_1\neq\gamma_2$.
\item $P(T)=a_5(T-\beta)^3(T-\gamma)(T-\overline{\gamma})$ with $\beta\in\R$ and $\gamma\in\C-\R$.
\item $P(T)=a_5(T-\beta)^2(T-\gamma_1)^2(T-\gamma_2)$ with $\beta,\gamma_1,\gamma_2\in\R$ and $\beta\neq\gamma_1\neq\gamma_2$.
\item $P(T)=a_5(T-\beta)(T-\gamma)^2(T-\overline{\gamma})^2$ with $\beta\in\R$ and $\gamma\in\C-\R$.
\end{enumerate}
\item If $s_0(P)=0$, $s_1(P)\neq 0$  then $\tau_1(T)={\bf Sres}_{1}(P)$. The only possible cases are:
\begin{enumerate}
\item $P(T)=a_5(T-\beta)^2(T-\gamma_1)(T-\gamma_2)(T-\gamma_3)$ with $\beta,\gamma_1,\gamma_2,\gamma_3\in\R$ and $\beta\neq\gamma_1\neq\gamma_2\neq\gamma_3$.
\item $P(T)=a_5(T-\beta)^2(T-\gamma_1)(T-\gamma_2)(T-\overline{\gamma_2})$ with $\beta,\gamma_1\in\R$ and $\gamma_2\in\C-\R$.
\end{enumerate}
\end{enumerate}

In order to separate cases 2(a) and 2(b), we start noting that $\tau_1(T)={\bf Sres}_{3}(\tau_0)$. In case 2(a) we have
$$\tau_1(T)=s_3(\tau_0)(T-\beta)^3$$
and in case 2(b) we have
$$\tau_1(T)=s_3(\tau_0)(T-\beta)^2(T-\gamma).$$
Subresultants $s_0(\tau_1)$ and $s_1(\tau_1)$ separate these two cases:
\begin{itemize}
\item if $s_0(\tau_1)=s_1(\tau_1)=0$ then $\tau_1(T)=s_3(\tau_0)(T-\beta)^3$ and $P(T)=a_5(T-\beta)^4(T-\gamma)$.
\item if $s_0(\tau_1)=0$ and $s_1(\tau_1)\neq 0$ then $\tau_1(T)=s_3(\tau_0)(T-\beta)^2(T-\gamma)$ and $P(T)=a_5(T-\beta)^3(T-\gamma)^2$.
\end{itemize}
In case 2(a) we have 
$$\beta=-\frac{s_{3,2}(\tau_0)}{3s_3(\tau_0)}=-\frac{s_{3,2}(P)}{3s_3(P)},\quad 
T-\gamma=\frac{\tau_0(T)}{a_5(T-\beta)^4}=\frac{P(T)}{a_5(T-\beta)^4}\quad\hbox{and}\quad
\gamma=-\frac{a_0}{a_5\beta^4}.$$
And, in case 2(b), we have 
$$\tau_2(T)={\bf Sres}_{1}(\tau_1)=s_1(\tau_1)(T-\beta)$$
and
$$\beta=-\frac{s_{1,1}(\tau_1)}{s_1(\tau_1)},\quad 
T-\gamma=\frac{\tau_1(T)}{s_3(\tau_0)(T-\beta)^2}\quad\hbox{and}\quad
\gamma=-\frac{s_{3,0}(P)}{s_3(P)\beta^2}.$$

Cases 3(a), 3(b), 3(c) and 3(d) are separated by the degree of $\tau_2(T)$ and the signs of $a_5$, $s_3(P)$ and $s_2(P)$ in the following way: 
\newpage

\begin{itemize}
\item Case 3(a): $\deg(\tau_2(T))=1$ and ${\bf C}(\{a_5,5a_5,s_3(P),s_2(P),0,0\})=3$.
\item Case 3(b): $\deg(\tau_2(T))=1$ and ${\bf C}(\{a_5,5a_5,s_3(P),s_2(P),0,0\})=1$.
\item Case 3(c): $\deg(\tau_2(T))=0$ and ${\bf C}(\{a_5,5a_5,s_3(P),s_2(P),0,0\})=3$.
\item Case 3(d): $\deg(\tau_2(T))=0$ and ${\bf C}(\{a_5,5a_5,s_3(P),s_2(P),0,0\})=1$.
\end{itemize}
Cases 3(a) and 3(c) requiere $a_5$, $s_3(P)$ and $s_2(P)$ to have the same sign. Formulae for $\beta$ and the other real roots of $P(T)$ for these four cases follow:
\begin{itemize}
\item Case 3(a): $$\beta=-\frac{s_{2,1}(\tau_0)}{2s_1(\tau_0)}=-\frac{s_{1,1}(\tau_1)}{s_1(\tau_1)}\quad\hbox{and}\quad 
(T-\gamma_1)(T-\gamma_2)=\frac{P(T)}{a_5(T-\beta)^3}.$$
\item Case 3(b): $$\beta=-\frac{s_{2,1}(\tau_0)}{2s_1(\tau_0)}=-\frac{s_{1,1}(\tau_1)}{s_1(\tau_1)}\quad\hbox{and}\quad 
(T-\gamma)(T-\overline{\gamma})=\frac{P(T)}{a_5(T-\beta)^3}.$$
\item Case 3(c): $$(T-\beta)(T-\gamma_1)=\frac{\tau_1(T)}{s_2(P)},\quad
T-\gamma_2=\frac{s_2(P)^2P(T)}{a_5\tau_1(T)^2}
\quad\hbox{and}\quad
\gamma_2=\frac{s_2(P)^2a_0}{a_5s_{3,0}(P)^2}.
$$
\item Case 3(d): $$T-\beta=\frac{P(T)}{a_5({\bf Sres}_{2}(P))^2} \quad\hbox{and}\quad \beta=-\frac{a_0}{a_5s_{2,0}^2}.$$
\end{itemize}

In both cases 4(a) and 4(b), we have
$$\beta=-\frac{s_{1,1}(P)}{s_1(P)}$$
and they are separated by analysing the signs of $a_5$, $s_3(P)$, $s_2(P)$ and $s_1(P)$:
\begin{itemize}
\item Case 4(a): ${\bf C}(\{a_5,5a_5,s_3(P),s_2(P),s_1(P),0\})=4$.
\item Case 4(b):  ${\bf C}(\{a_5,5a_5,s_3(P),s_2(P),s_1(P),0\})=2$.
\end{itemize}
In case 4(a) we have
$$(T-\gamma_1)(T-\gamma_2)(T-\gamma_3)=\frac{P(T)}{a_5(T-\beta)^2}$$
and in case 4(b) we have
$$(T-\gamma_1)(T-\gamma_2)(T-\overline{\gamma_2})=\frac{P(T)}{a_5(T-\beta)^2}.$$
Both polynomials have no multiple roots and we know that they have exactly three and one different real roots respectively.

In all cases we have  characterized the multiple real roots of any degree $5$ polynomial with multiple roots as explicit functions of its coefficients. These functions are either rational functions in those coefficients or rational functions involving the square root of a polynomial, in those coefficients known to be strictly positive. Simple real roots of these polynomials (with multiple roots) are also characterized in the same way with the exception of cases 4(a) and 4(b) where these real roots ($0$ or $3$) come from a cubic equation (without multiple roots). 

\subsection{$\deg(P)=6$ or $\deg(P)=7$}
When the degree of $P$ is $6$ or $7$, not in all cases we can describe the multiple real roots as rational functions of the coefficients of $P$, or rational functions involving the square root of a polynomial in those coefficients known to be strictly positive, like before. The unique cases where the strategy followed for degrees $4$ and $5$ fails are the following:
\begin{itemize}
\item $\deg(P)=6$:
\begin{itemize}
\item $P(T)=a_6(T-\gamma_1)^2(T-\gamma_2)^2(T-\gamma_3)^2$ with $\gamma_1,\gamma_2,\gamma_3\in\R$ and $\gamma_1\neq\gamma_2\neq\gamma_3$.
\item $P(T)=a_6(T-\gamma_1)^2(T-\gamma_2)^2(T-\overline{\gamma_2})^2$ with $\gamma_1\in\R$ and $\gamma_2\in\C-\R$.
\end{itemize}
\item $\deg(P)=7$:
\begin{itemize}
\item $P(T)=a_7(T-\gamma_1)^2(T-\gamma_2)^2(T-\gamma_3)^2(T-\gamma_3)$ with $\gamma_1,\gamma_2,\gamma_3,\gamma_4\in\R$ and $\gamma_1\neq\gamma_2\neq\gamma_3\neq\gamma_4$.
\item $P(T)=a_7(T-\gamma_1)^2(T-\gamma_2)(T-\gamma_3)^2(T-\overline{\gamma_3})^2$ with $\gamma_1,\gamma_2\in\R$, $\gamma_1\neq\gamma_2$ and $\gamma_3\in\C-\R$.
\end{itemize}
\end{itemize}
In all remaining cases, we can proceed like before concerning multiple real roots. In all cases, we can compute a polynomial without multiple roots whose real roots are the simple real roots of the considered polynomial.

\subsection{$\deg(P)=n$}
The analysis presented for degrees 4, 5, 6 and 7 can be generalized as presented in the next theorem.

\begin{theorem} 
Let $P(T)$ be a polynomial in $\R[T]$ factorizing in the following way:
$$P(T)=\prod_{i=1}^{r}(T-\beta_i)^{m_i}\prod_{i=r+1}^{r+s}\left((T-\gamma_i)(T-\overline{\gamma_i})\right)^{m_i}
\prod_{k=1}^{t}(T-\delta_k)\prod_{k=t+1}^{t+q}(T-\phi_k)(T-\overline{\phi_k})
=\sum_{\ell=0}^{n} a_{\ell}T^{\ell}$$
with $m_i>1$ and $\beta_i,\delta_k\in\R$ and $\gamma_i,\phi_k\in\C-\R$ (all of them different). Then:
\begin{enumerate}
\item If there are no repetitions in $m_1, m_2, \ldots,m_{r+s}$, then every $\beta_i$ can be described explicitly like a rational function of the $a_{\ell}$'s.
\item If the repeated elements in $m_1, m_2, \ldots,m_{r}$ appear at most twice and every $m_i$, $1\leq i\leq r$, does not appear in  $m_{r+1}, m_{r+2}, \ldots,m_{r+s}$ then every $\beta_i$ can be described explicitly like a rational function of the $a_{\ell}$'s involving in some cases the square root of a polynomial in the $a_{\ell}$'s known to be strictly positive.
\end{enumerate}
\end{theorem} 

\begin{proof}
Without loss of generality, we assume $m_1\geq m_2\geq \ldots\geq m_{r}$ and $m_{r+1}\geq m_{r+2}\geq \ldots\geq m_{r+s}$. We proceed like in the case $\deg(P)=5$ by defining $\tau_0(T)=P(T)$ and, for $k\geq 1$: 
$$\tau_k(T)=\gcd(\tau_{k-1},\tau_{k-1}').$$

If there are no repetitions in $m_1, m_2, \ldots,m_{r+s}$ and $m_1>m_{r+1}$, then $\deg(\tau_{m_1-1}(T))=1$ and $\beta_1$ is the unique root of $\tau_{m_1-1}(T)$. If $m_1<m_{r+1}$, then $\deg(\tau_{m_1-1}(T))=2$ and $\gamma_1$ and  $\overline{\gamma_1}$ are the roots of $\tau_{m_1-1}(T)$. Since $\tau_1(T)$ is one of the subresultants of $P$, $\tau_2(T)$ is one of the subresultants of $\tau_1(T)$, and so on,  we can conclude that the coefficients of $\tau_{m_1-1}(T)$ are polynomials in the $a_{\ell}$'s and that $\beta_1$ can be described explicitly like a rational function of the $a_{\ell}$'s. The first part follows either after removing $(\tau_{m_1-1}(T))^{m_1}$ from $P(T)$ and repeating the same process. 

When $\tau_k(T)$ is computed, we know the subresultants of $\tau_j(T)$, $0\leq j\leq k-1$ so that we can shorten this process once we know that the number of different real roots of $\tau_k(T)$ is equal to $1$ or $0$: in the first case we continue until the corresponding $\tau_k(T)$ has only one real root and no imaginary roots (this happens when $\tau_{k+1}(T)$ has no real roots and this information is provided directly by the subresultants producing $\tau_{k+1}(T)$); in the second case we eliminate $(\tau_k(T))^{k+1}$ from $P(T)$ since we are only interested in the real roots of $P(T)$.

If the repeated elements in $m_1, m_2, \ldots,m_{r}$ appear at most twice and every $m_i$, $1\leq i\leq r$, does not appear in  $m_{r+1}, m_{r+2}, \ldots,m_{r+s}$, then we proceed like in the previous case. The only difference appears when dealing with $\beta_1$ and $\beta_2$ such that $m=m_1=m_2>m_j$: in this case we stop when computing the  $2$ degree polynomial $\tau_{m-1}(T)$, whose roots are $\beta_1$ and $\beta_2$. When the biggest multiplicity appears in the ``complex" side, then we proceed in the same way until we reach $\tau_k(T)$ with no real roots:  in this case, again, we eliminate $(\tau_k(T))^{k+1}$ from $P(T)$, since we are only interested in the real roots of $P(T)$.
\end{proof}

\begin{remark} \hskip 0pt\\ {\rm 
The proof of the previous theorem can be done in a more direct way by using the squarefree decomposition of $P(T)$. The included proof provides the algorithm producing the desired description for the multiple real roots of $P(T)$ and the polynomial without multiple roots whose real roots are the simple real roots of $P(T)$.
}\end{remark}

\section{Avoiding general position: branch computations around critical points}\label{sec:branching}

In order to analyse the topology of ${\cal C}_P$ when there are only one critical point in a critical line it is very easy to determine how many half-branches there are to the left and to the right of the considered critical point. When there are more than one critical point the situation becomes more complicated but having an explicit  and easy to manipulate description of the considered critical point allows to determine the required information about the half-branches.

If $(\alpha,\beta)\in\R^2$ is a  singular point of ${\cal C}_P$ then $$P(\alpha,\beta)={\frac{\partial P}{\partial y}}(\alpha,\beta)= {\frac{\partial P}{\partial x}}(\alpha,\beta)= 0.$$

\subsection{Ramification points of ${\cal C}_P$}\label{branching_criticalnosingular}
Recall that $(\alpha,\beta)\in\R^2$ is a critical and non singular point of ${\cal C}_P$ when $P(\alpha,\beta)=P_y(\alpha,\beta)=0$ and $P_x(\alpha,\beta)\neq 0.$
Applying the Implicit Function Theorem this means that around $(\alpha,\beta)$ the curve ${\cal C}_P$ can be described as a function $x=\Phi(y)$ such that $\alpha=\Phi(\beta)$. Since 
$$\Phi'(\beta)=-\frac{P_y(\alpha,\beta)}{P_x(\alpha,\beta)}=0$$
we have three possibilities, since $y=\beta$ can be: 
\begin{itemize}
\item a local minimun of $\Phi$: $2$ half-branches to the right of $(\alpha,\beta)$, $0$ to the left; or 
\item a local maximun of $\Phi$: $0$ half-branches to the right of $(\alpha,\beta)$, $2$ to the left; or
\item an inflection point of $\Phi$: $1$ half-branch to the right of $(\alpha,\beta)$ and $1$ to the left. 
\end{itemize}
Characterising the behaviour of the function $x=\Phi(y)$ at $y=\beta$ requires to evaluate the derivatives $P_{yy}(\alpha,\beta),P_{yyy}(\alpha,\beta),P_{yyyy}(\alpha,\beta),\ldots$ until one of them does not vanishes since: if 
$\Phi^{(1)}(\beta)=\Phi^{(2)}(\beta)=\cdots=\Phi^{(k-1)}(\beta)=0, \quad \Phi^{(k)}(\beta)\neq 0$,
then (applying recursively implicit differentiation to $P(\Phi(y),y)=0$), we have that
$\Phi^{(k)}(\beta)P_x(\alpha,\beta)=-P_{y{\buildrel{k}\over\cdots}y}(\alpha,\beta)$
and we can apply the higher order derivative test to determine whether the point is a local maximum, a local minimum, or a flex. 


\subsection{Singular points of ${\cal C}_P$: $\deg_y(P)=4$}\label{branching_4}

Due to the low degree of the curve, the only ambiguity that can arise happens when there are two double roots $\beta_1<\beta_2$ of the polynomial $P(\alpha,y)$. Otherwise it is locally general position for the critical line $x=\alpha$. Since we work with a quartic, we have, at most 4 branches to each of the sides of our critical line. Since we have found two double roots for $P(\alpha,y)$, we know that the coefficient of $y^4$ must be nonzero. We'll suppose it is 1 for simplicity. Moreover, since $P$ is defined over the reals, the number of real branches to each side of the critical line must be even. 

If we have four branches to join with the singular points to one of the sides, then it must be two for each due to multiplicity. 
If there are no real branches to one of the sides
, we have no work to do for such branch.

Finally, the tricky case is when we have just two branches to join. These two branches  must go to the same point, since the other critical point must attract two conjugate complex branches. First of all, we consider $Q_i(s,t)=P(s+\alpha,t+\beta_i)$. Then the behavior of $(0,0)$ as a point for $Q_i$ is the same as the behavior of $(\alpha,\beta_i)$ for $P$. Factoring the lowest homogeneous component of $Q_i$ we have the slopes of the (at most two) tangent lines to $\mathcal{C}_P$ at $(\alpha,\beta_i)$. Then:
\begin{itemize}
\item If one of the point has all slopes to be complex non real, it is an isolated point, so the other takes the branches.
\item If one of the points has two different real slopes, it takes the two arcs, since it is a real node.
\end{itemize}

In the case that there is just one slope for the tangent lines to the curve at the critical points, we will consider the cubic curve given by $P_y(x,y)=0$. The polynomial $P_y(\alpha,y)$ vanishes in $\beta_1$, $\beta_2$ and an intermediate point $\gamma\in (\beta_1,\beta_2)$ since it is the derivative of $P(\alpha,y)$. This means that there are three real branches of $\mathcal{C}_{P_y}$ through the vertical line $x=\alpha$. Due to the low degree, the only posibility is what happens in Figure \ref{branches_deg_4} or the symmetric case, and the relative position of the branches of $\mathcal{C}_P$ and $\mathcal{C}_{P_y}$ determines how to join the half branches.

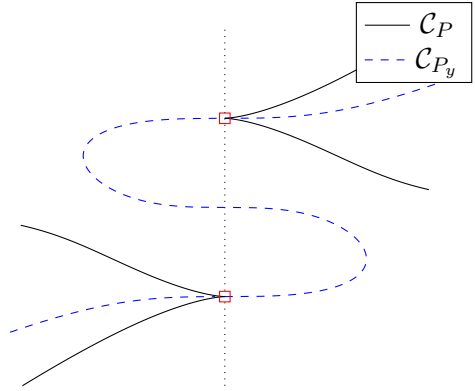
\begin{figure}[hbt]
\centering
\begin{tikzpicture}
\begin{axis}[
    axis lines = none,
    ]

\addplot[
	domain=0.2:1,
	color=black,
	samples=300
]({((x^4-2*x^2+1)/x)^(1/3)},{x});
\addlegendentry{$\mathcal{C}_P$};

\addplot[
	dashed,
	domain=1:1.4,
	color=blue,
	samples=300
]({(4*x^3-4*x)^(1/3)},{x});
\addlegendentry{$\mathcal{C}_{P_y}$};

\addplot[
	domain=0.2:1,
	color=black,
	samples=300
]({-((x^4-2*x^2+1)/x)^(1/3)},{-x});

\addplot[
	domain=1:2,
	color=black,
	samples=300
]({((x^4-2*x^2+1)/x)^(1/3)},{x});

\addplot[
	domain=1:2,
	color=black,
	samples=300
]({-((x^4-2*x^2+1)/x)^(1/3)},{-x});

\addplot[
	dashed,
	domain=-1:0,
	color=blue,
	samples=300
]({(4*x^3-4*x)^(1/3)},{x});

\addplot[
	dashed,
	domain=0:1,
	color=blue,
	samples=300
]({-(4*x-4*x^3)^(1/3)},{x});

\addplot[
	dashed,
	domain=-1.4:-1,
	color=blue,
	samples=300
]({-(4*x-4*x^3)^(1/3)},{x});

\addplot[mark=none, black, dotted] coordinates {(0,-2) (0,2)};
\addplot[mark=square, red] coordinates {(0,-1)};
\addplot[mark=square, red] coordinates {(0,1)};
\end{axis}
\end{tikzpicture}
\caption{\label{branches_deg_4} To the left, the real branches of $\mathcal{C}_P$ go to the below critical point because two branches of $\mathcal{C}_{P_y}$ are above them both. To the right, we have the complementary situation.}
\end{figure}


\subsection{Singular points of ${\cal C}_P$: $\deg_y(P)=5$}

We now consider deg$_yP=5$. The nontrivial cases here are:
\begin{itemize}
\item $P(\alpha,y)$ has two double roots.
\item $P(\alpha,y)$ has a triple root and a double root.
\end{itemize}
We will address each case separately, but first we consider, as before, that $P$ is monic on $y$ (and deg$_y(P)=5$, otherwise, we proceed as in lower degree). Then, reasoning in a similar way, we see that the number of real branches between critical lines must be 1, 3 or 5.

\subsubsection{$P(\alpha,y)$ has two double roots.}

This case can be treated as in the case of degree 4. We have two critical points that take either two or no branches each, and one single point that wil take one.

If we have 5 branches to distribute, each critical point takes two and the non-critical point takes one. We distribute the branches to avoid crossings outside the critical line.

If we have just one branch, the non-critical point takes it.

If we have three branches, one goes to the noncritical point and the other two are assigned either checking whether the tangent lines are real and different at the singularities,  or considering the curve $\mathcal{C}_{P_y}$ and reasoning as in the degree 4 case:
\begin{itemize}
\item If there are at least two branches of $\mathcal{C}_{P_y}$ above at least two of the three branches of $\mathcal{C}_P$, then the critical point below takes two branches.
\item 
Otherwise, the critical point above takes two branches.
\end{itemize}  

\subsubsection{$P(\alpha,y)$ has a triple root and a double root.}

Here, the triple critical point takes 1 or three branches, and the double one takes two or none.

If there are five branches to distribute, three go to the critical point corresponding to the triple root, and two go to the critical point corresponding to the double root.

If there is just one branch, the critical point corresponding to the triple root takes it.

If there are three branches, we again check the slopes of the tangent lines:
\begin{itemize}
\item If one of the critical points has two complex non-real slopes, the other one takes the until now unassigned branches.
\item If one of the critical points has (at least) two real slopes, it takes the until now unassigned branches.
\end{itemize}
If we do not have enough data, then we consider again $\mathcal{C}_{P_y}$. It has one real branch through the double point, one real branch passing between the critical points and two possibly non-real branches passing through the triple point.
\begin{itemize}
\item If $\mathcal{C}_{P_y}$ has just two real branches, the double point takes the until now unassigned branches.
\item If two of the four real branches of $\mathcal{C}_{P_y}$ lie above the three branches of  $\mathcal{C}_{P}$, the below critical point takes the until now unassigned branches.
\item Otherwise, the above critical point takes the until now unassigned branches.
\end{itemize}
It is impossible that the three branches of $\mathcal{C}_{P}$ lie between the four branches of $\mathcal{C}_{P_y}$ with this configuration at the critical line.

\subsection{Singular points of ${\cal C}_P$: the general case}

 We can generalise what we have got with the slopes of the tangent lines at critical points for general degree. It is well known that a tangent line to a curve at a singular point $(\alpha,\beta)$ corresponds to a branch through it. Therefore, if the lowest degree homogeneous component $h(s,t)$ of $P(s+\alpha,t+\beta)$ is square free and not a multiple of $s$, each linear factor of $h(s,t)$ in $\mathbb{R}[s,t]$ corresponds to a real branch through the point, and each quadratic factor corresponds to two conjugate non-real branches. This solves the problem for general singularities without vertical tangents.

\section{Characterising the intersection curve between two ellipsoids}\label{sec:applications}
 
Given two ellipsoids $\mathcal{A}: XAX^{T}=0$ and $\mathcal{B}: XBX^{T}=0$, $X=(x,\,y,\,z,\,1)$, their characteristic equation (or polynomial) is defined as $$f(\lambda)={\rm det}(\lambda A+B)=\det(A)\lambda^4+\ldots+\det(B)$$ which is a quartic polynomial
in $\lambda$ with real coefficients. The characterization of the relative position of two ellipsoids in terms of the sign of the real roots of their characteristic equation was introduced by \cite{WWK}.

\begin{theorem}\label{basic_ellipsoids}\hskip 0pt\\
Let ${\cal A}$ and ${\cal B}$ be two ellipsoids with the characteristic equation
$f(\lambda)$. Then:
\begin{enumerate}
\item The characteristic equation $f(\lambda)$ always has at least two negative roots.
\item ${\cal A}$ and ${\cal B}$ are separated if and only if $f(\lambda)$ has two distinct positive roots.
\item ${\cal A}$ and ${\cal B}$ touch each other externally if and only if $f(\lambda)$ has a positive double root.
\end{enumerate}
\end{theorem} 

With this theorem, we can decide the basic relative positions, i.e., separation, externally touching and overlapping, of two
ellipsoids by the real root pattern of the characteristic polynomial of the quadric pencil formed by these two ellipsoids
However, the root pattern of the characteristic polynomial is not enough to characterize the arrangement of two ellipsoids.

A more in-depth algebraic characterization using the so-called index sequence was introduced in \cite{TWMW} to classify the morphology of the intersection curve of two quadratic surfaces in the the 3D real projective space. The index sequence of a quadric pencil not only includes the root pattern of the characteristic polynomial, but also involves the Jordan form associated to each root and the information between two consecutive roots. The index sequence requires to define the index function of a quadric pencil.

The behaviour of the index function for a pencil of ellipsoids is captured by the eigenvalue curve ${\cal S}$ defined by the equation $S(\lambda,\mu)=\det(\lambda A + B - \mu\I_4)=0\ .$
$S$ has degree four in both $\lambda$  and $\mu$. Because $\lambda A + B$ is a real symmetric matrix for each $\lambda\in\R$, there are in total four real roots for $S(\lambda,\mu)=0$, counting 
multiplicities. For each value $\lambda_0$, the index function $Id(\lambda_0)$ equals 
to the number of positive real roots of $S(\lambda_0,\mu)=0$.

Since $S(\lambda,\mu)=0$ is a very special quartic curve (there are always four real branches (taking into account multiplicities) its analysis is extremely simple. If $\lambda=\alpha$ is a critical line then $S(\alpha,\mu)$ factorizes in the following way:
\begin{enumerate}
\item $S(\alpha,\mu)=\tau_4(y-\beta)^4$.
\item $S(\alpha,\mu)=\tau_4(y-\beta)^3(y-\gamma)$ with $\gamma\in \R$.
\item $S(\alpha,\mu)=\tau_4(y-\beta)^2(y-\gamma)^2$ with $\gamma\in \R$.
\item $S(\alpha,\mu)=\tau_4(y-\beta)^2(y-\gamma_1)(y-\gamma_2)$ with $\gamma_1\neq \gamma_2$.
\end{enumerate}

In \ref{subsec:multipleroots_4} we can find formulae showing, in  terms of $\alpha$, the values of $\beta$, $\gamma$, $\gamma_1$ and $\gamma_2$ allowing to determine easily $Id(\alpha)$. Computing $Id(\lambda)$ for $\lambda$ not giving a critical line reduce to apply Descartes' law of signs (see Remark 2.38 in \cite{BPR}) to the polynomial $S(\lambda,\mu)$ as polynomial in $\mu$. The way the four branches touch every critical line is easily determined by using the techniques described in \ref{branching_4}.

\section{Conclusions}\label{sec:conclusions}
In this paper we have introduced a family of formulae describing the multiple roots of a univariate polynomial equation like rational functions of the coefficients of the considered polynomial. These formulae have been used to try to avoid the use of the ``general position condition" when computing the topology of a real algebraic plane curve defined implicitly. A concrete application has been also described and next step will be to design a new algorithm computing the topology of an arrangement of quartics and quintics by using the formuale and strategy introduced here.

\section{\bf Acknowledgements. }
The authors are partially supported by the grant PID2020-113192GB-I00/AEI/ 10.13039/501100011033 (Mathematical Visualization: Foundations, Algorithms and Applications) from the Spanish State Research Agency (Ministerio de Ciencia e Innovación). J. Caravantes belongs to the Research Group ASYNACS (Ref. CT-CE2019/683).

\end{document}